\documentclass{amsart}

\usepackage{amsmath}
\usepackage{amssymb}
\usepackage{amsthm}

\DeclareMathOperator{\am}{am}

\DeclareMathOperator{\dist}{dist}
\DeclareMathOperator{\Am}{Am}

\newtheorem{theorem}{Theorem}[section]
\newtheorem{lemma}[theorem]{Lemma}

\newtheorem{corollary}[theorem]{Corollary}

\theoremstyle{definition}

\theoremstyle{remark}
\newtheorem{remark}[theorem]{Remark}

\numberwithin{equation}{section}

\begin{document}

\title[A geometric characterization of range-kernel complementarity]{A geometric characterization of range-kernel complementarity}
\author{Dimosthenis Drivaliaris}
\address{Department of Financial and Management Engineering\\
University of the Aegean\\
Chios, 82100\\
Greece}
\email{d.drivaliaris@aegean.gr}
\author{Nikos Yannakakis}
\address{Department of Mathematics\\
School of Mathematical and Physical Sciences\\
National Technical University of Athens\\
Iroon Polytexneiou 9\\
15780 Zografou\\
Greece}
\email{nyian@math.ntua.gr}
\commby{}

\begin{abstract}
We show that a bounded linear operator on a Banach space with closed range has range-kernel complementarity if and only if its generalized amplitude is less than $\pi$. An application to the strong convergence of the iterations of a bounded linear operator is also given.
\end{abstract}

\maketitle
\section{Introduction}

Let $X$ be a complex Banach space and $A:X\rightarrow X$ be a bounded linear operator on $X$. If we denote by $R(A)$ and $N(A)$ the range and the kernel of $A$ respectively and $R(A)$ is closed, then range-kernel complementarity is the decomposition
\begin{equation}
\label{rkc}
X=R(A)\oplus N(A)\,.
\end{equation}
In a sense, range-kernel complementarity is the closest thing to $A$ being invertible, since it is equivalent to $A$ being of the form ``invertible$\,\oplus\, 0$''. Our aim in this short paper is its characterization by a simple geometric condition.

Our approach combines two ingredients. The first is that (\ref{rkc}) is possible if and only if $A$ can be factorized as $TP$ with $T$ invertible and $P$ a projection, see \cite[Theorem 3]{mbehkta}. The second is that (\ref{rkc}) follows from a suitable angle of $A$ being less than $\pi$, see \cite[Theorem 3.4]{drivyann}.

The idea behind the possible connection between range-kernel complementarity and some geometric condition on $A$ originates from the following question. If an angle condition is imposed on $A$ (for example, the well-known $Re\langle Ax,x\rangle\geq c\|x\|^2$ of the Lax-Milgram Theorem which implies that $A$ is invertible) how far from invertible can $A$ be? The answer is interesting and surprising; $A$ cannot be too bad as a quite natural angle condition implies (\ref{rkc}).

We conclude this paper with an application of our results to the problem of strong convergence of the iterates $T^n$ of an asymptotically regular operator $T$, for which $R(I-T)$ is closed.

For other characterizations of range-kernel complementarity we refer the interested reader to \cite{mbehkta} and the references therein.

\section{Main result}

In order to formulate our geometric condition we first define the angle of $A$ with respect to a certain class of linear operators.
This class $\mathcal{E}_A$ for $A$ is defined as follows
\[\mathcal{E}_A=\left\{T\in B(X): T(N(A))=N(A)\text{ and }A+tT\text{ is invertible for some }t>0\right\}\,.\]

\begin{remark}
Note that in the terminology of \cite[p. 49]{gohberg} if $T\in\mathcal{E}_A$, then $A+tT,\,\,\,t\in\mathbb{C}$, is a regular linear pencil.
\end{remark}

If $[\cdot,\cdot]$ is a semi-inner product in $X$ compatible with its norm, then the cosine of $A$ with respect to $T\in\mathcal{E}_A$ is defined as
\[\cos_T(A)=\inf\left\{\frac{Re[Ax,Tx]}{\|Ax\|\,\|Tx\|}: x\notin N(A)\cup N(T)\right\}\]
and the angle of $A$ with respect to $T$ is defined as
\[\varphi_T(A)=\arccos\left(\cos_T(A)\right)\,.\]

For $T=e^{i\theta}I\,,\,\,\,\theta\in[0,2\pi]$, the above is nothing more than the angle of $A$ along the ray
\[\rho_\theta=\left\{0\right\}\cup\left\{z\in\mathbb C: \arg z=\theta\right\}\]
which was defined by Gustafson and Krein in \cite{gustafson} and \cite{krein} respectively. For each $\theta$ this angle $\varphi_\theta(A)$ measures the maximum turning effect of $A$ along the ray $\rho_\theta$.

The amplitude of $A$, also introduced by Krein in \cite{krein}, is then defined as
\[\am(A)=\min\left\{\phi_T(A):\,T=e^{i\theta}I,\,\,\,\theta\in [0,2\pi]\right\}\,.\]
The amplitude compares the maximum turning effect of $A$ along every possible $\rho_\theta$ and provides the smallest one.

In \cite{drivyann} we proved that if $A$ has closed range, $R(A)+N(A)$ is closed and $\am(A)<\pi$, then $A$ has range-kernel complementarity. The converse of this is not true, since it may be proved that $\am(A)<\pi$ implies that the spectrum of $A$ is confined to a sector of the complex plane. In \cite{drivyann1} the hypothesis that $R(A)+N(A)$ is closed was dropped and a characterization of range-kernel complementarity was given, using the angle of $A$ along a curve, for closed range operators whose spectrum is not confined to a sector of $\mathbb C$, but still does not separate $0$ from $\infty$. Hence, again, no complete characterization of range-kernel complementarity was given.

Trying to amend these shortcomings and motivated by the amplitude of Krein we define the generalized amplitude of $A$ with respect to $\mathcal{E}_A$ to be
\[\Am(A)=\inf\left\{\varphi_T(A): T\in\mathcal{E}_A\right\}\,.\]

\begin{remark}
(a) Obviously $\Am(A)\leq \am(A)$, for all operators $A$.\\
(b) In the special case of an invertible operator $A$, we have that $A\in\mathcal{E}_A$ and hence $\Am(A)=\varphi_A(A)=0$. This simple example encapsulates a crucial difference between the generalized amplitude and Krein's one. The bilateral shift $A$ on $l^2$ is an invertible operator and so $\Am(A)=0$, whereas its Krein amplitude is equal to $\pi$, since its spectrum separates 0 from $\infty$. Note that by \cite[Theorem 4.1]{drivyann1} the angle of $A$ along any curve is also equal to $\pi$.\\
(c) More generally, if $A:X\rightarrow X$ is invertible, then for the non-invertible operator $A\oplus 0:X\oplus X\rightarrow X\oplus X$ we have that $A\oplus I\in\mathcal{E}_A$ and hence $\Am(A\oplus 0)\leq\frac{\pi}{2}$. Taking $A$ to be the bilateral shift we get an example of a non-invertible operator with $\Am(A)<\am(A)$.
\end{remark}

Using this generalized amplitude we get the following characterization of range-kernel complementarity.

\begin{theorem}
\label{main}
Let $X$ be a Banach space and $A:X\rightarrow X$ be a bounded linear operator with closed range. Then
\[X=R(A)\oplus N(A)\]
if and only if $\Am(A)<\pi$.
\end{theorem}

To get to the proof we start with five preparatory Lemmas.

\begin{lemma}
\label{estimate}
If $\varphi_T(A)<\pi$, then there exists $c>0$ such that $\|Ax+tTx\|\geq c\|Ax\|$, for all $x\notin N(A)$ and all $t\geq 0$.
\end{lemma}

\begin{proof}
We have that $\varphi_T(A)<\pi$ implies that there exists $\delta>0$ such that
\[\frac{Re[Ax,Tx]}{\|Ax\|\,\|Tx\|}\geq -1+\delta\,,\]
for all $x\notin N(A)\cup N(T)$ and hence
\[Re[Ax,Tx]+\|Ax\|\,\|Tx\|\geq \delta\,\|Ax\|\,\|Tx\|\,,\]
for all $x\notin N(A)\cup N(T)$\,.

So for $t\geq 0$ and $x\notin N(A)\cup N(T)$ we get that
\begin{align*}
&Re[Ax+tTx,Tx]+\|Ax+tTx\|\,\|Tx\|\geq\\
&\,\,\,\,\,\,\,\,\,\,\,\,\,\,\,\geq Re[Ax,Tx]+t\|Tx\|^2 +(\|Ax\|-t\|Tx\|)\|Tx\|\\
&\,\,\,\,\,\,\,\,\,\,\,\,\,\,\,=Re[Ax,Tx]+\|Ax\|\,\|Tx\|\\
&\,\,\,\,\,\,\,\,\,\,\,\,\,\,\,\geq \delta\|Ax\|\,\|Tx\|\,.
\end{align*}
Thus, for $c=\frac{\delta}{2}$, we get $\|Ax+tTx\|\geq c\|Ax\|$, for all $x\notin N(A)\cup N(T)$ and $t\geq 0$. Since this inequality obviously holds if $x\in N(T)$ the proof is complete.
\end{proof}

\begin{lemma}
\label{uniform}
Assume $A\in B(X)$ has closed range and $\varphi_T(A)<\pi$, for some $T\in B(X)$. If $N(A)=\left\{0\right\}$ and $A+t_0T$ is invertible, for some $t_0>0$, then $A+tT$ is invertible, for any $|t-t_0|<\mu$, where $\mu>0$ is independent of $t_0$.
\end{lemma}

\begin{proof}
Since the range of $A$ is closed, there exists $k_A>0$ such that
\[\|Ax\|\geq k_A\dist(x,N(A))\]
and thus, since $N(A)=\left\{0\right\}$,
\[\|Ax\|\geq k_A\|x\|\,,\text{ for all }x\in X\,.\]
Let $\mu=c\,k_A>0$, where $c$ is the constant from Lemma \ref{estimate}. Since $N(A)=\left\{0\right\}$, we get that
\begin{equation}
\label{lowerbound}
\|Ax+tTx\|\geq \mu\|x\|\,,\text{ for all }x\in X \text{ and }t\geq 0\,.
\end{equation}
So $A+tT$ is injective for all $t\geq 0$\,.

Hence it is enough to show that if  $|t-t_0|<\mu $ then $A+tT$ is also onto. So fix such a $t$ and let $L=\frac{|t-t_0|}{\mu}<1$. For $y$ in $X$, define $F:X\rightarrow X$ by
\[Fh=(A+t_0T)^{-1}\left(y+(t_0-t)h\right),\text{ for every }h\in X\,.\]
Then, using (\ref{lowerbound}), we have
\begin{align*}
\|Fh_1-Fh_2\|&=\|(A+t_0T)^{-1}\left(y+(t_0-t)h_1\right)-(A+t_0T)^{-1}\left(y+(t_0-t)h_2\right)\|\\
&\leq\frac{1}{\mu}\,\|(t_0-t)(h_1-h_2)\|\\
&=L\,\|h_1-h_2\|\,.
\end{align*}
Hence $F$ is a contraction and thus there exists a unique $x\in X$ such that $Fx=x$, which implies that $(A+t_0T)x=y+(t_0-t)x$. Hence $(A+tT)x=y$ and so $A+tT$ is onto as was required.
\end{proof}

Next we show that if some angle of $A$ is less than $\pi$, then the sum of $R(A)$ and $N(A)$ is both closed and direct.

\begin{lemma}
\label{closed}
Assume $\varphi_T(A)<\pi$, for some $T\in \mathcal{E}_A$. Then $R(A)+N(A)$ is closed and $R(A)\cap N(A)=\left\{0\right\}$.
\end{lemma}

\begin{proof}
If $x\notin N(A)$, $y\in N(A)$ and $t>0$, from Lemma \ref{estimate} we get that
\[\left\|A\left(x+\frac{1}{t}\,y\right)+tT\left(x+\frac{1}{t}\,y\right)\right\|\geq c\|Ax\|\]
and thus
\[\|Ax+tTx+Ty\|\geq c\|Ax\|\,.\]
Letting $t\rightarrow 0$ we have that
\[\|Ax+Ty\|\geq c\|Ax\|\,.\]
Since $T(N(A))=N(A)$, we get that the sum $R(A)+N(A)$ is closed.

To conclude note that if $Ax\in N(A)$, then from the above inequality $Ax=0$ and so $R(A)\cap N(A)=\left\{0\right\}$.
\end{proof}

The final two results that we will need for the proof of Theorem \ref{main} are the following:

\begin{lemma}
\label{descent}{\cite[Corollary p. 129]{mbehkta}}
If
\[R(A)\cap N(A)=\left\{0\right\}\,,\]
then $R(A)+N(A)$ is closed if and only if $R(A^2)$ is closed.
\end{lemma}

\begin{lemma}{\cite[Corollary 5.2]{drivyann1}}
\label{projection}
If $P$ is a projection, then $\varphi_I(P)<\pi$.
\end{lemma}

We may now proceed with the proof of our main result.

\begin{proof}[Proof of Theorem \ref{main}]
Assume that $X=R(A)\oplus N(A)$. Then, by \cite[Theorem 3]{mbehkta}, there exist an invertible operator $S$ and a projection $P$ such that $A$ can be factorized as $A=SP$. In particular $S=A|_{R(A)}\oplus I_{N(A)}$ with respect to $X=R(A)\oplus N(A)$ and $P$ is the projection onto $R(A)$ parallel to $N(A)$. Hence $S(N(A))=N(A)$ and $SP=PS$ and so, by Lemma \ref{projection}, we get that
\[\varphi_S(A)=\varphi_I(P)<\pi\,.\]
Thus $\Am(A)<\pi$.

For the converse if $\Am(A)<\pi$, then there exists $T\in\mathcal{E}_A$ such that $\varphi_T(A)<\pi$.

We first assume that $N(A)=\left\{0\right\}$. Since $T\in\mathcal{E}_A$ we have that $A+t_0T$ is invertible for some $t_0>0$ and hence, by Lemma \ref{uniform}, we get that so is $A+tT$, for $|t-t_0|<\mu$. Since $\mu>0$ is independent of $t_0$, continuing this way we get that $A$ is invertible and hence we have range-kernel complementarity.

If $N(A)\neq\left\{0\right\}$ and since by Lemma \ref{closed} we have that $R(A)\cap N(A)=\left\{0\right\}$, in order to conclude (see \cite[Proposition 38.4]{heuser}) it is enough to show that $R(A^2)=R(A)$. By Lemma \ref{descent} we have that $R(A^2)$ is closed and hence if
\[R\left(A|_{R(A)}\right)=R(A^2)\neq R(A)\,,\]
then $A|_{R(A)}$ has trivial kernel, closed range but is not invertible. Hence by the first part of the proof
\[\varphi_T(A|_{R(A)})=\pi\,.\]
But $\varphi_T(A|_{R(A)})\leq \varphi_T(A)$ and thus we have a contradiction. Hence
$R(A^2)=R(A)$ and the proof is complete.
\end{proof}

\begin{remark}
If the operator $T$ for which $\varphi_T(A)<\pi$ is invertible, then things are much simpler. In particular then $\varphi(T^{-1}A)<\pi$ and hence, by \cite[Theorem 3.4]{drivyann}, we have that
\[X=R(T^{-1}A)\oplus N(T^{-1}A)\,.\]
Therefore by \cite[Theorem 3]{mbehkta} there exist an invertible operator $S$ and a projection $P$ such that $T^{-1}A=SP$ and thus $A=TSP$ which, again by \cite[Theorem 3]{mbehkta}, gives
\[X=R(A)\oplus N(A)\]
as required.
\end{remark}

If $0$ is an isolated point of the spectrum of $A$, then with respect to the Riesz decomposition $A$ is of the form ``invertible$\,\oplus$ quasinilpotent'' . Since range-kernel complementarity is equivalent to the ``quasinilpotent part'' of $A$ being $0$ it is quite expected that quasinilpotent operators with closed range should have amplitude equal to $\pi$.

\begin{corollary}
\label{quasinilpotent}
Let $X$ be a Banach space and $Q$ be a quasinilpotent operator with closed range. Then $\Am(Q)=\pi$.
\end{corollary}

\begin{proof}
If $\Am(Q)<\pi$ then by Theorem \ref{main}
\[X=R(Q)\oplus N(Q)\]
which is impossible since by \cite{grabiner} the range chain of a quasinilpotent operator does not terminate.
\end{proof}

\begin{remark}
As $\Am(Q)\leq \am(Q)$, for quasinilpotent operators with closed range the generalized amplitude is equal to Krein's. Moreover since $\sigma(Q)$ does not separate 0 from $\infty$ by \cite[Theorem 4.1]{drivyann1} the angle of $Q$ along any curve is also equal to $\pi$.
\end{remark}

\section{An application}

In this section we present a simple application of our main result to the problem of strong convergence of the iterates $T^n$ of a bounded linear operator $T$ defined on a Banach space $X$. Recall that $T$ is called asymptotically regular if, for all $x\in X$,
\[(T^n-T^{n+1})x\rightarrow 0,\text{ as }n\rightarrow +\infty\,.\]

\begin{remark}
If $A$ is nilpotent then it is asymptotically regular and hence the asymptotic regularity of $A$ does not imply that $\Am(A)<\pi$.
\end{remark}

The proof of our result follows the idea of \cite[Theorem 4.1]{badea}.

\begin{theorem}
\label{alt}
Let $X$ be a complex Banach space and $T\in B(X)$ be an asymptotically regular operator, such that $R(I-T)$ is closed. The sequence of iterates $T^nx$ converges, for all $x\in X$, as $n\rightarrow +\infty$ if and only if $\Am(I-T)<\pi$.
\end{theorem}

\begin{proof}
Since $R(I-T)$ is closed if $\Am(I-T)<\pi$ we have from Theorem \ref{main} that
\[X=R(I-T)\oplus N(I-T)\,.\]
Moreover by the fact that $T$ is asymptotically regular we have that $T^nz\rightarrow 0$, as $n\rightarrow +\infty$, for all $z\in R(I-T)$. Hence if $x\in X$, then $x=y+z$ with $y\in N(I-T)$ and $z\in R(I-T)$ and so
\[T^nx=T^ny+T^nz\rightarrow y\text{ as }n\rightarrow +\infty\,.\]

The converse follows again from Theorem \ref{main} since the strong convergence of $T^n$ implies by \cite[Theorem 3]{halperin} that $X=R(I-T)\oplus N(I-T)$.
\end{proof}

\begin{remark}
(a) It may be easily seen from the above proof that $T^n$ converges strongly to the projection onto $R(I-T)$ parallel to $N(I-T)$.\\
(b) Note that, by \cite[Theorem 4.1]{badea}, $R(I-T)$ is closed if and only if the convergence of $T^n$ is uniform.\\
(c) The connection between range-kernel complementarity and the strong convergence of the iterates $T^n$ was used in the well-known paper of Halperin \cite{halperin} in order to prove the convergence of the iterates of $T=P_1P_2...P_n$, for $P_1,...,P_n$ orthogonal projections on a Hilbert space. This was later generalized for norm-one projections on a Banach space by Bruck and Reich in \cite[Theorem 2.1]{reich} and by Badea and Lyubich in \cite[Main Theorem, p. 25]{badea}, with suitable additional conditions .
\end{remark}

Recall that we say that a contraction $T$ is primitive if $T$ satisfies the Katznelson-Tzafriri spectral condition, i.e.\ if
\[\sigma(T)\subseteq\left\{1\right\}\cup\left\{t\in \mathbb C: |t|<1\right\}\,.\]
We have the following Corollary of Theorem \ref{alt} for primitive contractions on a reflexive Banach space.

\begin{corollary}
\label{primitive}
If $X$ is a reflexive Banach space and $T$ is a primitive contraction with $R(I-T)$ closed, then $\Am(I-T)<\pi$.
\end{corollary}

\begin{proof}
To see this note that $T$ is asymptotically regular and by \cite[Theorem 4.1]{badea} its iterates $T^n$ converge strongly. Hence the result follows by Theorem \ref{alt}.
\end{proof}

\begin{remark}
(a) Since $T$ in Corollary \ref{primitive} is a contraction $\sigma(I-T)$ does not separate 0 from $\infty$ and hence, by \cite[Theorem 4.1]{drivyann1}, we have that the Krein amplitude of $I-T$ is also less than $\pi$.\\
(b) Corollary \ref{primitive} combined with \cite[Main Theorem]{badea} shows the following: If $T$ is a convex combination of products of orthoprojections on a complex Banach space $X$, $R(I-T)$ is closed, $X$ is uniformly convex or uniformly smooth or reflexive, and the orthoprojections are of class (D) (for the definition see \cite{badea}), then $\Am(I-T)<\pi$. We would like to note here that since such a $T$ may be quasinilpotent with closed range, the generalized amplitude of such a $T$, by Corollary \ref{quasinilpotent}, may be equal to $\pi$.
\end{remark}


\end{document}